\documentclass{article}
\pdfoutput=1
\usepackage{amsthm}
\usepackage{graphicx}

\usepackage{amsmath}
\usepackage{amssymb}
\usepackage{amsfonts}

\makeatletter
  {\par\addvspace{\@bls \@plus 0.5\@bls \@minus 0.1\@bls}\noindent
   {\bfseries\Elproofname}\enspace\ignorespaces}%
  {\par\addvspace{\@bls \@plus 0.5\@bls \@minus 0.1\@bls}}
\makeatother

\usepackage[utf8]{inputenc}
\usepackage{url}
\usepackage{multicol}
\usepackage{multirow}
\usepackage{enumerate}
\usepackage{float}
\floatstyle{plain}

\usepackage{booktabs}
\usepackage{xspace}
\usepackage{mathtools, cuted}
\usepackage{paralist} 

\usepackage{algorithm}
\usepackage{algorithmic}

\makeatletter
\renewcommand{\ALG@name}{Algorithme}

  \newcommand{\IFTHEN}[3][default]{\ALC@it\algorithmicif\ #2\
    \algorithmicthen\ #3\
    \ifthenelse{\boolean{ALC@noend}}{}{\algorithmicendif\ } \ALC@com{#1}}
  \newcommand{\IFTHENNOEND}[3][default]{\ALC@it\algorithmicif\ #2\
    \algorithmicthen\ #3\
    \ALC@com{#1}}
  \newcommand{\ONLYELSE}[1][default]{\ALC@it\algorithmicelse%
\ALC@com{#1}\begin{ALC@if}}
\newcommand{\IFTHENELSE}[3]{\ALC@it\algorithmicif\ #1\
  \algorithmicthen\ #2\
  \algorithmicelse\ #3\
  \ifthenelse{\boolean{ALC@noend}}{}{\algorithmicendif\ } }
\newcommand{\IFTHENEND}[2]{\ALC@it\algorithmicif\ #1\
  \algorithmicthen\ #2\ \algorithmicendif}
\makeatother


\usepackage{todonotes}

\DeclareMathOperator{\rank}{rank}

\newtheorem{theorem}{Theorem}
\newtheorem{remark}{Remark}

\newtheorem{definition}{Definition}
\newtheorem{lemma}{Lemma}

\newtheorem{problem}{Problem}
\newcommand{\F}{\ensuremath{\mathbb{F}}}
\newcommand{\Z}{\ensuremath{\mathbb{Z}}}

\newcommand{\RPM}[1]{\ensuremath{\mathcal{R}_{#1}}\xspace}

\newcommand{\TRMM}{\texttt{TRMM}\xspace}

\newcommand{\TRSM}{\texttt{TRSM}\xspace}

\newcommand{\PERM}{\texttt{PERM}\xspace}
\newcommand{\GEMM}{\texttt{GEMM}\xspace}
\newcommand{\SCAL}{\texttt{SCAL}\xspace}
\newcommand{\DADD}{\texttt{DADD}\xspace}
\newcommand{\PLDUQ}{\texttt{PLDUQ}\xspace}
\newcommand{\SYRDK}{\texttt{SYRDK}\xspace}
\newcommand{\SYRDtwoK}{\texttt{SYRD2K}\xspace}
\newcommand{\TRSS}{\texttt{TRSSYR2K}\xspace}

\usepackage{mathtools}

\makeatletter

\DeclareRobustCommand\vdots{%
  \mathpalette\@vdots{}%
}
\newcommand*{\@vdots}[2]{%
  \sbox0{$#1\cdotp\cdotp\cdotp\m@th$}%
  \sbox2{$#1.\m@th$}%
  \vbox{%
    \dimen@=\wd0 %
    \advance\dimen@ -3\ht2 %
    \kern.5\dimen@
    \dimen@=\wd2 %
    \advance\dimen@ -\ht2 %
    \dimen2=\wd0 %
    \advance\dimen2 -\dimen@
    \vbox to \dimen2{%
      \offinterlineskip
      \copy2 \vfill\copy2 \vfill\copy2 %
    }%
  }%
}
\DeclareRobustCommand\ddots{%
  \mathinner{%
    \mathpalette\@ddots{}%
    \mkern\thinmuskip
  }%
}
\newcommand*{\@ddots}[2]{%
  \sbox0{$#1\cdotp\cdotp\cdotp\m@th$}%
  \sbox2{$#1.\m@th$}%
  \vbox{%
    \dimen@=\wd0 %
    \advance\dimen@ -3\ht2 %
    \kern.5\dimen@
    \dimen@=\wd2 %
    \advance\dimen@ -\ht2 %
    \dimen2=\wd0 %
    \advance\dimen2 -\dimen@
    \vbox to \dimen2{%
      \offinterlineskip
      \hbox{$#1\mathpunct{.}\m@th$}%
      \vfill
      \hbox{$#1\mathpunct{\kern\wd2}\mathpunct{.}\m@th$}%
      \vfill
      \hbox{$#1\mathpunct{\kern\wd2}\mathpunct{\kern\wd2}\mathpunct{.}\m@th$}%
    }%
  }%
}

\newenvironment{smatrix}{\left[\begin{smallmatrix}}{\end{smallmatrix}\right]}

\newcommand{\mat}[1]{\ensuremath{\mathbf{#1}}}

\makeatletter
\renewcommand*\env@matrix[1][*\c@MaxMatrixCols c]{%
  \hskip -\arraycolsep
  \let\@ifnextchar\new@ifnextchar
  \array{#1}}
\makeatother

\newcommand{\OpenDreamKit}{the \href{http://opendreamkit.org}{OpenDreamKit} \href{https://ec.europa.eu/programmes/horizon2020/}{Horizon 2020} \href{https://ec.europa.eu/programmes/horizon2020/en/h2020-section/european-research-infrastructures-including-e-infrastructures}{European Research Infrastructures} project (\#\href{http://cordis.europa.eu/project/rcn/198334_en.html}{676541})}

\title{Symmetric indefinite triangular factorization revealing the rank profile matrix\footnote{This work is partly funded by \OpenDreamKit.}}

\author{Jean-Guillaume Dumas\footnote{
  {Universit\'e Grenoble Alpes}.
  {Laboratoire Jean Kuntzmann, CNRS, UMR 5224}.
  {700 avenue centrale, IMAG - CS 40700},
  {38058 Grenoble, cedex 9}
  {France}.
\href{mailto:Jean-Guillaume.Dumas@univ-grenoble-alpes.fr,Clement.Pernet@univ-grenoble-alpes.fr}{\{firstname.lastname\}@univ-grenoble-alpes.fr}}
\and Cl\'ement Pernet\footnotemark[2]}

\usepackage{color,svgcolor}
\usepackage[plainpages=true]{hyperref}
\makeatletter
\hypersetup{
  pdfauthor={Jean-Guillaume Dumas and Cl\'ement Pernet},
  breaklinks=true,
  colorlinks=true,
  linkcolor=darkred,
  citecolor=blue,
  urlcolor=darkgreen,
}
\makeatother

\begin{document}

\maketitle

\begin{abstract}
We present a novel recursive algorithm for reducing a symmetric matrix to a
triangular factorization which reveals the rank profile matrix.  
That is, the algorithm computes a factorization  $\mat{P}^T\mat{A}\mat{P} = \mat{L}\mat{D}\mat{L}^T$ where $\mat{P}$ is a
permutation matrix, $\mat{L}$ is lower triangular with a unit diagonal and  $\mat{D}$ is
symmetric block diagonal with $1{\times}1$ and $2{\times}2$
antidiagonal blocks.
The novel algorithm requires $O(n^2r^{\omega-2})$ arithmetic operations.
Furthermore, experimental results demonstrate that our algorithm can even
be slightly more than twice as fast as the state of the art unsymmetric Gaussian
elimination  in most cases, that is it achieves
approximately the same computational speed. 
By adapting the pivoting strategy developed in the unsymmetric case, we show
how to recover the rank profile matrix from the permutation matrix and the
support of the block-diagonal matrix.
There is an obstruction in characteristic $2$ for revealing the rank profile
matrix which requires to relax the shape of the block diagonal by allowing the
2-dimensional blocks to have a non-zero bottom-right coefficient.
This relaxed decomposition can then be transformed into a standard 
$\mat{P}\mat{L}\mat{D}\mat{L}^T\mat{P}^T$
decomposition at a negligible cost.
\end{abstract}

\section{Introduction}

Computing a triangular factorization of a symmetric matrix is a commonly used
kernel to solve symmetric linear systems, or to compute the signature of
symmetric bilinear forms.
Besides the fact that it is expected to save half of the arithmetic cost of a standard
(non-symmetric) Gaussian elimination, it can also recover invariants, such as
the signature, specific to symmetric matrices, and thus, e.g., be used to
certify positive or negative definite or
semidefiniteness~\cite[Corollary~1]{Kaltofen:2011:quadcert}.

It is a fundamental computation in numerical linear algebra, and is therefore most
often presented in the setting of real matrices. When the matrix is positive
definite, the Cholesky factorization can be defined: $\mat{A}=\mat{L}\mat{L}^T$, where $\mat{L}$ is
lower triangular for which square roots of diagonal elements have to be
extracted. Alternatively, gathering the diagonal elements in a central diagonal
matrix yields the LDLT factorization $\mat{A}=\mat{L}\mat{D}\mat{L}^T$ which no longer requires square
roots. Similarly as for the LU decomposition, it is only defined for matrices
with generic rank profile, i.e. having their $r=\text{rank}(\mat{A})$ first
leading principal minors non-zero. For arbitrary matrices, symmetric
permutations may lead to the former situations:
$\mat{P}\mat{A}\mat{P}^T=\mat{L}\mat{D}\mat{L}^T$. However, this is
unfortunately not always the case. For instance there is no permutation
$\mat{P}$ such that $\begin{smatrix} 0&1\\1&0\end{smatrix}$ has a LDLT
factorization with a diagonal $\mat{D}$.
This lead to a series of generalizations where the matrix $\mat{D}$ was replaced first
by a tridiagonal symmetric matrix by Parlett and Reid~\cite{Parlett:1970:SSL},
improved by Aasen~\cite{Aasen:1971:tridiag}, achieving half the arithmetic cost
of Gaussian elimination. Bunch and Kaufman then replaced this tridiagonal matrix
by a block diagonal composed of~1 or 2-dimensional diagonal blocks.

\paragraph{Pivoting}
In numerical linear algebra, the choice of the permutation matrix is mainly
driven by the need to ensure a good numerical quality of the decomposition.
Bunch and Parlett~\cite{Bunch:1971:DMS} use a full pivoting technique,
requiring a cubic number of tests. Bunch and Kaufman pivoting strategy,
implemented in LAPACK, uses a partial pivoting requiring only a quadratic number
of tests.

In the context of exact linear algebra, for instance when computing over a finite
field, numerical stability is no longer an issue. However, the computation of
echelon forms and rank profiles, central in many applications, impose further
constraints on the pivoting. A characterization of the requirements for the
pivoting strategy is given in~\cite{DPS15,jgd:2017:bruhat} so that a PLUQ
decomposition can reveal these rank profiles and echelon forms in the
non-symmetric case.
In particular, it is shown that pivot selection minimizing the lexicographic order
on the coordinate of the pivot, combined with row and column rotations to move
the pivot to the diagonal, enable the computation of the rank profile matrix,
an invariant from which all rank profile information, the row and the column
echelon form can be recovered.

\paragraph{Recursive algorithms}
As in numerical linear algebra, we try to gather arithmetic operations in level
3 BLAS operations (matrix multiplication based), for it delivers the best
computation throughput.
Numerical software often use tiled implementations, especially when the pivoting is more
constrained by the symmetry~\cite{Shklarski:2007:blockrec,Elmroth:2004:rec},
or in order to define communication avoiding variants~\cite{BBDDDPSTY14}.
In exact linear algebra sub-cubic matrix multiplication, such as Strassen's
algorithm, can be extensively used with no numerical instability issues. This
led to the design of recursive algorithms, which was proven successful
in the unsymmetric case, including for shared memory parallel computations~\cite{DGPRZ16}. 





\paragraph{Contribution}
The contribution here is to propose a recursive algorithm producing a symmetric
factorization PLDLTPT  over any field,
from which the rank profile matrix of the input can be recovered. This algorithm
is a recursive variant of Bunch and Kaufman's algorithm~\cite{Bunch:1977:Kaufman} where the pivoting
strategy has been replaced by the one developped previously by the authors in
the unsymmetric case~\cite{jgd:2017:bruhat}. Compared to the 
recursive adaptation of Aasen's algorihtm in~\cite{Rozloznik:2011:PTT}, our
algorithm leads to a similar data partitionning but does not suffer from an
arithmetic overhead compared to Aasen's algorithm.
Our algorithms has time complexity $O(n^2r^{\omega-2})$ where $\omega$ is an
admissible exponent for matrix multiplication and $r$ is the rank of the input
matrix.
With $\omega=3$, the leading constant in the time complexity is $1/3$, matching
that of the best alternative algorithms based on cubic time linear algebra.


In Section~\ref{sec:rpm} we show that in characteristic two the rank profile
matrix can not always be revealed by a symmetric factorization with antidiagonal
blocks: sometimes antitriangular blocks are also required.  
Then we recall in Section~\ref{sec:tools} the main required level 3 linear algebra
subroutines. In Section~\ref{sec:algo} we present the main recursive
algorithm. An alternative iterative Crout variant is 
presented in Section~\ref{sec:basecase} to be used as a base case in the recursion.
We finally show, in Section~\ref{sec:expe}, experiments of the resulting
implementation over a finite field. They demonstrate the efficiency of cascading
the recursive algorithm with the base case variant, especially with matrices involving a
lot of pivoting. They finally confirm a speed-up by a factor of about 2 compared to the state of the art
unsymmetric Gaussian elimination.

\section{The symmetric rank profile matrix}\label{sec:rpm}
\subsection{The pivoting matrix}
Theorem~\ref{th:rpm} recalls the definition of the rank profile matrix.
\begin{theorem}[\cite{DPS15}]\label{th:rpm}
   Let $\mat{A}\in\F^{m\times n}$. There exists a unique $m\times n$
 $\{0,1\}$-matrix $\RPM{A}$ with  $r$ 1's in rook placement
 of which every leading sub-matrix has the same
 rank as the corresponding
 leading sub-matrix of $\mat{A}$.
 This matrix is called the \em{rank profile matrix} of~$\mat{A}$.
\end{theorem}

\begin{lemma}
A symmetric matrix has a symmetric rank profile matrix.
\end{lemma}
\begin{proof}
  Otherwise, the rank of some leading submatrix of $\mat{A}$ and the same
  leading submatrix of $\mat{A}^T$ would be different which is absurd.
\end{proof}

Also, any symmetric matrix has a triangular decomposition
$\mat{A}=\mat{P}\mat{L}\mat{D} \mat{L}^T \mat{P}^T$ where $\mat{L}$ is unit lower triangular, $\mat{D}$ is block diagonal,
formed by 1-dimensional scalar blocks or 2-dimensional blocks of the form
$\begin{smatrix}  0&x\\x&0\end{smatrix}$ and $\mat{P}$ a permutation matrix.

We here further define ${\Psi}$ as the support matrix of $\mat{D}$: namely, a
block 
diagonal $\{0,1\}$-matrix such that $\mat{D}={\Psi} \overline{\mat{D}}$, with $\overline{\mat{D}}$ a
diagonal matrix. 
  \begin{definition}
    The pivoting matrix of a PLDLTPT decomposition is the matrix $\Pi=\mat{P}{\Psi} \mat{P}^T$.
  \end{definition}

  \begin{definition}
    A PLDLTPT reveals the rank profile matrix of a symmetric matrix $\mat{A}$ if its
    pivoting matrix equals the rank profile matrix of $\mat{A}$.
  \end{definition}

\subsection{Antitriangular blocks in characteristic two}
In zero or odd characteristic, we show next that one can always find such a
PLDLTPT decomposition revealing the rank profile matrix. In  characteristic two,
however, this is not always possible.
  \begin{lemma}\label{lem:cex}
    In characteristic 2, there is no symmetric indefinite elimination revealing the rank
  profile matrix of $\mat{A}=\begin{bmatrix}0&1\\1&1\end{bmatrix}$.
  \end{lemma}
\begin{proof}
Let $\mat{J}$ be the $2{\times}2$ anti-diagonal identity matrix. This is also the rank
profile matrix of $\mat{A}$.
Now, we let $\mat{L}=\begin{bmatrix}1&0\\x&1\end{bmatrix}$,
$\overline{\mat{D}}=\begin{bmatrix}y&0\\0&z\end{bmatrix}$. As the permutation matrices
involved, $P$ and ${\Psi}$, can only be either the identity matrix or $\mat{J}$, there
are then four cases: 
\begin{enumerate}
\item $\mat{A}=\mat{L}{\cdot}\overline{\mat{D}}{\cdot}\mat{L}^T=\begin{bmatrix}y&x y\\x y&x^2 y+z\end{bmatrix}$,
  but $y=0$ and $\rank(\overline{\mat{D}})=\rank(A)=2$ are incompatible.
\item $\mat{A}=\mat{J}{\cdot}\mat{L}{\cdot}\overline{\mat{D}}{\cdot}\mat{L}^T{\cdot}\mat{J}^T=\begin{bmatrix}x^2y+z&x y\\x y&y\end{bmatrix}$, but $\mat{J}{\cdot}\mat{I}{\cdot}\mat{J} \neq \mat{J}=\RPM{\mat{A}}$.
\item $\mat{A}=\mat{L}{\cdot}\overline{\mat{D}}{\cdot}\mat{J}{\cdot}\mat{L}^T=\begin{bmatrix}0&y\\z&xy+xz\end{bmatrix}$,
  but we need $y=z$ for the symmetry and then $2xy=0\neq 1$ in
  characteristic $2$. 
\item
  $\mat{A}=\mat{J}{\cdot}\mat{L}{\cdot}\overline{\mat{D}}{\cdot}\mat{J}{\cdot}\mat{L}^T{\cdot}\mat{J}^T=\begin{bmatrix}xy+xz&z\\y&0\end{bmatrix}$ but the bottom right coefficient of $\mat{A}$ is non zero. 
\end{enumerate}%
\end{proof}

However, one can generalize the PLDLTPT decomposition to a block diagonal
matrix $\mat{D}$ having 2-dimensional blocks of the form
$ \begin{smatrix}   0&c\\c&d  \end{smatrix}$ (lower antitriangular).
Then the support matrix $\Psi$ of
$\mat{D}$ is the block diagonal $\{0,1\}$ matrix such that $\mat{D}=\Psi \overline{\mat{D}}$, with
$\overline{\mat{D}}$ an upper triangular bidiagonal matrix (or equivalently such that
$\mat{D}=\overline{\mat{D}}\Psi$, with $\overline{\mat{D}}$ lower triangular bidiagonal). 

With these generalized definitions, we show in Section~\ref{sec:algo}, that
there exists RPM-revealing PLDLTPT decompositions.\\

\subsection{Antitriangular decomposition}\label{ssec:antiti}
Then, such a generalized decomposition can always be further reduced to a strict
PLDLTPT decomposition by eliminating each of the antitriangular blocks.  
For this, the observation is that in
characteristic two, a symmetric lower antitriangular $2{\times}2$ block is
invariant under any symmetric triangular transformation:
$\begin{bmatrix}1&\\x&1\end{bmatrix}
\begin{bmatrix}&c\\c&d\end{bmatrix}
\begin{bmatrix}1&x\\&1\end{bmatrix}=
\begin{bmatrix}&c\\c&2cx+d\end{bmatrix}\equiv
\begin{bmatrix}&c\\c&d\end{bmatrix}\mod 2=
\begin{bmatrix}1&\\&1\end{bmatrix}
\begin{bmatrix}&c\\c&d\end{bmatrix}
\begin{bmatrix}1&\\&1\end{bmatrix}$.
Thus for each $2{\times}2$ block in a tridiagonal decomposition, the
corresponding $2{\times}2$ diagonal block in $\mat{L}$ can be replaced by $\mat{I}_2$, via a
multiplication by $\begin{smatrix}1&\\-x&1\end{smatrix}$.

Further, we have that: 
$\begin{smatrix}&c\\c&d\end{smatrix}=
\mat{J}
\begin{smatrix}1&\\c/d&1\end{smatrix}
\begin{smatrix}d&\\&-c^2/d\end{smatrix}
\begin{smatrix}1&c/d\\&1\end{smatrix}
\mat{J}$.
Now $\mat{J}$ commutes with the identity $\mat{I}_2$ matrix. 
Therefore we have that:
$\begin{smatrix}1&\\x&1\end{smatrix}
\begin{smatrix}1&\\-x&1\end{smatrix}
\mat{J}
\begin{smatrix}1&\\c/d&1\end{smatrix}
=
\mat{J}
\begin{smatrix}1&\\x&1\end{smatrix}
\begin{smatrix}1&\\-x&1\end{smatrix}
\begin{smatrix}1&\\c/d&1\end{smatrix}
$.

Thus, to eliminate the antitriangular blocks, 
create a triangular matrix $\mat{L}_J$ that starts as the identity
and where its $i,i+1$ blocks corresponding to a
$\begin{smatrix}1&\\x&1\end{smatrix}$ block in $\mat{L}$  is a
$\begin{smatrix}1&\\c/d-x&1\end{smatrix}$ block (associated to an
antitriangular $\begin{smatrix}&c\\c&d\end{smatrix}$ block, with $d{\neq}0$, in
$\mat{D}$).
Then replace the triangular
matrix $\mat{L}$ by $\tilde{\mat{L}}=\mat{L}{\cdot}\mat{L}_J$.
Also, modify the diagonal matrix $\mat{D}$, to $\tilde{\mat{D}}$ such that 
the $\begin{smatrix}&c\\c&d\end{smatrix}$ blocks of $\mat{D}$ are replaced by
$\begin{smatrix}d&\\&-c^2/d\end{smatrix}$ blocks in $\tilde{\mat{D}}$.
Finally, create a permutation matrix $\mat{P}_J$, starting from the identity matrix,
where each identity block at position $i,i+1$ corresponding to an antitriangular
block in $\mat{D}$ is replaced by $\mat{J}$. Then $\tilde{\mat{P}}=\mat{P}{\cdot}\mat{P}_J$.

From this we have now a symmetric PLDLTPT factorization,
$A=\tilde{\mat{P}}\tilde{\mat{L}}\tilde{\mat{D}}\tilde{\mat{L}}^T\tilde{\mat{P}}^T$, with purely $1{\times}1$
and $2{\times}2$ antidiagonal blocks in $\tilde{\mat{D}}$
(but then a direct access to the rank profile matrix, $\mat{P} \Psi \mat{P}^T$, might not be
possible from $\tilde{\mat{P}}$ and $\tilde{\mat{D}}$). 

In the following we present some building blocks and then algorithms
computing RPM-revealing symmetric indefinite triangular factorization.

  \section{Building blocks}\label{sec:tools}

  We recall here some of the standard algorithms from the BLAS3~\cite{DdCHD90} and
  LAPACK~\cite{LAPACK99} interfaces and generalization thereof~\cite{BBD12}, which will be used to
  define the main block  recursive symmetric eliminating algorithm.

  \begin{description}
    \item[\GEMM$(\mat{C},\mat{A},\mat{B})$:] general matrix multiplication. Computes $\mat{C}\leftarrow \mat{C} - \mat{A}\mat{B}$. 
    \item[\TRMM$(\mat{U},\mat{B})$:]   multiply a triangular and a rectangular matrix in-place. Computes $\mat{B}\leftarrow \mat{U}\mat{B}$ where $\mat{B}$ is $m\times n$ and $\mat{U}$
      is upper or lower triangular.
    \item[\TRMM$(\mat{C},\mat{U},\mat{B})$:]  multiply a triangular and a rectangular
      matrix. Computes $\mat{C}\leftarrow \mat{C}- \mat{U}\mat{B}$ where $\mat{B}$ and $\mat{C}$ are $m\times n$ and $\mat{U}$
      is upper or lower triangular. This is an adaptation of the BLAS3 \TRMM to
      leave the $\mat{B}$ operand unchanged.
    \item[\TRSM$(\mat{U},\mat{B})$:]  solve a triangular system with matrix right
      hand-side. Computes $\mat{B}\leftarrow \mat{U}^{-1}\mat{B}$ where $\mat{B}$ is $m\times n$ and $\mat{U}$
      is upper or lower triangular.
    \item[\SYRDK$(\mat{C},\mat{A},\mat{D})$:] symmetric rank $k$ update with diagonal
      scaling. Computes the upper or lower triangular part of the symmetric
      matrix $\mat{C}\leftarrow \mat{C} - \mat{A}\mat{D}\mat{A}^T$ where $\mat{A}$ is $n\times k$ and $D$ is
      diagonal or block diagonal.
    \item[\SYRDtwoK$(\mat{C},\mat{A},\mat{D},\mat{B})$:] symmetric rank $2k$ update with diagonal
      scaling. Computes the upper or lower triangular part of the symmetric
      matrix $\mat{C}\leftarrow \mat{C} - \mat{A}\mat{D}\mat{B}^T-\mat{B}\mat{D}\mat{A}^T$ where $\mat{A}$ and $\mat{B}$ are $n\times k$, $\mat{B}$ and $\mat{D}$ is
      diagonal or block diagonal. 
  \end{description}

  In addition, we need to introduce the \TRSS routine solving  Problem~\ref{pb:trss}.
  \begin{problem}
    \label{pb:trss}
    Let $\F$ be a field of characteristic different than 2. Given a symmetric
    matrix $\mat{C}\in \F^{n\times n}$ and a unit upper triangular matrix
    $\mat{U}\in\F^{n\times n}$, find an upper triangular matrix $X\in \F^{n\times n}$
    such that $\mat{X}^T\mat{U}+\mat{U}^T\mat{X}=\mat{C}$.
  \end{problem}
In characteristic 2, the diagonal of $X^T\mat{U}+\mat{U}^TX$ is always zero for any matrix $X$ and
$\mat{U}$, hence Problem~\ref{pb:trss} has no solution as soon as $\mat{C}$ has a non-zero
diagonal element. 

However in characteristic zero or odd, Algorithm~\ref{alg:trss} presents a
recursive implementation of this 
routine, and is in the same time a constructive proof of the existence of such a
solution. Note that it performs a division by 2 in line~\ref{line:halving}, and
therefore requires that the base field has not characteristic two.
\begin{algorithm}[htb]
\caption{\TRSS($\mat{U},\mat{C}$)}\label{alg:trss}
\begin{algorithmic}[1]
  \REQUIRE{ $\mat{U}$, $n \times n$ full-rank upper triangular}
  \REQUIRE{ $\mat{C}$, $n \times n$, symmetric}
  \ENSURE{  $\mat{C}\leftarrow \mat{X}$ where $\mat{X}$ is $n \times n$ upper triangular, such that $\mat{X}^T\mat{U}+\mat{U}^T\mat{X}=\mat{C}$.}
  \IF {$m=1$}
  \STATE $ \mat{C}_{1,1}\leftarrow \frac{1}{2}\mat{C}_{1,1}{\cdot}\mat{U}_{1,1}^{-1} $
  ; \label{line:halving} \textbf{return}
  \ENDIF
    \STATE Splitting $\mat{C}= \begin{smatrix} \mat{C}_1&\mat{C}_2\\\mat{C}_2^T&\mat{C}_3 \end{smatrix}$, $U= \begin{smatrix} \mat{U}_1&\mat{U}_2\\&\mat{U}_3 \end{smatrix}$ where
      $\mat{C}_1$ and $\mat{U}_1$ are $\left\lfloor \frac{n}{2} \right\rfloor\times \left\lfloor \frac{n}{2} \right\rfloor$.
  \STATE 
  \STATE Find $\mat{X}_1$ s.t. $\mat{X}_1^T\mat{U}_1+\mat{U}_1^T\mat{X}_1=\mat{C}_1$\;\hfill\COMMENT{\TRSS($\mat{U}_1,\mat{C}_1$)}
  \STATE $\mat{D}_2\leftarrow{}\mat{C}_2-\mat{X}_1^T\mat{U}_2$\;\hfill\COMMENT{\TRMM($\mat{X}_1^T,\mat{U}_2$)}
  \STATE $\mat{X}_2\leftarrow{}\mat{U}_1^{-T}\mat{D}_2$\;\hfill\COMMENT{\TRSM($\mat{U}_1^T,\mat{D}_2$)}
  \STATE $\mat{D}_3\leftarrow{}\mat{C}_3-(\mat{X}_2^T\mat{U}_2+\mat{U}_2^T\mat{X}_2)$\;\hfill\COMMENT{\SYRDtwoK($\mat{X}_2,\mat{U}_2$)}
  \STATE Find $\mat{X}_3$ s.t. $\mat{X}_3^T\mat{U}_3+\mat{U}_3^T\mat{X}_3=\mat{D}_3$\;\hfill\COMMENT{\TRSS($\mat{U}_3,\mat{D}_3$)}
\end{algorithmic}
\end{algorithm}

\begin{remark} Note that algorithm~\ref{alg:trss} computes the solution $X$ in
  place on the symmetric storage of $C$: by induction $\mat{X}_1$ and $\mat{X}_3$ overwrite
  $\mat{C}_1$ and $\mat{C}_3$, and $\mat{X}_2$ overwrites $\mat{C}_2$ according to the specifications of
  the generalized $\TRMM$ routine.
\end{remark}

\begin{lemma}\label{lem:trss}
  Algorithm \TRSS is correct and runs in $O(n^\omega)$ arithmetic operations.
\end{lemma}
\begin{proof}
  Using the notations of Algorithm~\ref{alg:trss}, let $\mat{X}=
  \begin{bmatrix}    \mat{X}_1&\mat{X}_2\\&\mat{X}_3  \end{bmatrix}$. Then exanding $\mat{X}^T\mat{U}+\mat{U}^T\mat{X}$ gives
  \begin{eqnarray*}
  \\
    \begin{bmatrix}
    \mat{X}_1^T\mat{U}_1+\mat{U}_1^T\mat{X}_1& \mat{X}_1^T\mat{U}_2 + \mat{U}_1^T\mat{X}_2\\
    (\mat{X}_1^T\mat{U}_2 + \mat{U}_1^T \mat{X}_2)^T &   \mat{X}_2^T\mat{U}_2+\mat{U}_2^T\mat{X}_2+\mat{X}_3^T\mat{U}_3+\mat{U}_3^T\mat{X}_3
    \end{bmatrix} \\
=    \begin{bmatrix}
    \mat{C}_1 & \mat{C}_2\\
    \mat{C}_2^T & \mat{C}_3-\mat{D}_3 +\mat{X}_3^T\mat{U}_3+\mat{U}_3^T\mat{X}_3
    \end{bmatrix}     =\mat{C}.
    \end{eqnarray*}
  which proves the correctin by induction.
  The arithmetic cost satisfy a recurrence of the form $T(n)=2T(n/2)+Cn^\omega$
  and is therefore $T(n)=O(n^\omega)$.
\end{proof}

\section{A block recursive algorithm}\label{sec:algo}

\subsection{Sketch of the recursive algorithm}\label{ssec:reccase}

The design  of a block recursive algorithm is based on the generalization of
the $2\times 2$ case into a block $2\times 2$ block algorithm. While scalars
could be either 0 or invertible, the difficulty in elimination algorithms, is
that a submatrix could be rank defficient but non-zero. We start here an
overview of the recursive algorithm by considering that the leading principal
block is either all zero or invertible. We will later give the general
presentation of the algorithm where its rank could be arbitrary.


Let  $\mat{M}\in\F^{(m+n){\times}(m+n)}$ be the symmetric matrix to be
factorized. Consider its block decomposition 
$\mat{M}=\begin{bmatrix} \mat{A} & \mat{B}\\ \mat{B}^T & \mat{C}\end{bmatrix}$ where $\mat{A}\in\F^{m\times m}$ and
$\mat{C}\in\F^{n\times n}$ are also symmetric. 

If $\mat{A}$ is full rank, then a recursive call will produce 
$\mat{A}=\mat{P}\mat{L}\mat{D}\mat{L}^TP^T$, and $\mat{M}$ can thus be decomposed as:
$$\mat{M}=
\begin{bmatrix}\mat{P}&\mat{0}\\\mat{0}&\mat{I}\end{bmatrix}
\begin{bmatrix}\mat{L}&\mat{0}\\\mat{G}&\mat{I}\end{bmatrix}
\begin{bmatrix}\mat{D}&\mat{0}\\\mat{0}&Z\end{bmatrix}
\begin{bmatrix}\mat{L}^T&\mat{G}^T\\\mat{0}&\mat{I}\end{bmatrix}
\begin{bmatrix}\mat{P}^T&\mat{0}\\\mat{0}&\mat{I}\end{bmatrix},$$
where $\mat{G}$ is such that $\mat{P}\mat{L}\mat{D}\mat{G}^T=\mat{B}$ and $\mat{Z}=\mat{C}-\mat{G}\mat{D}\mat{G}^T$.
Thus $\mat{G}$ can be computed as the transpose of
$\mat{D}^{-1}\mat{L}^{-1}\mat{P}^{-1}\mat{B}$ which can be
obtained by a call to \TRSM, some permutations and a diagonal scaling. Then
$\mat{Z}$ is computed by a call to \SYRDK. A second recursive call will then
decompose~$\mat{Z}$ and lead to the final factorization of $\mat{M}$.

Now if $\mat{A}$ is the zero matrix, one is reduced to factorize the matrix 
$\mat{N}=\begin{smatrix} \mat{0} & \mat{B}\\ \mat{B}^T & \mat{C}\end{smatrix}$. In order to recover the rank
profile matrix, one has to first look for pivots in $\mat{B}$ before considering the
block $\mat{C}$. Therefore diagonal pivoting is not an option here.
Then the matrix $\mat{B}$, which we assume has full rank for the moment, can be decomposed in a
$PLDUQ$ factorization ($\mat{P}$ and $\mat{Q}$ permutation matrices, $\mat{L}$ and $\mat{U}$ 
respectively unit lower and unit upper triangular, $\mat{D}$ is diagonal).
We then need to distinguish two cases depending on whether the field
characteristic is two or not.

\subsubsection{Zero or odd characteristic case}
If the characteristic zero or odd, $\mat{N}$ can thus be decomposed as:
$$\mat{N}=
\begin{bmatrix}\mat{P}&\mat{0}\\\mat{0}&\mat{Q}^T\end{bmatrix}
\begin{bmatrix}\mat{L}&\mat{0}\\\mat{G}&\mat{U}^T\end{bmatrix}
\begin{bmatrix}\mat{0}&\mat{D}\\\mat{D}&\mat{0}\end{bmatrix}
\begin{bmatrix}\mat{L}^T&\mat{G}^T\\\mat{0}&\mat{U}\end{bmatrix}
\begin{bmatrix}\mat{P}^T&\mat{0}\\\mat{0}&\mat{Q}\end{bmatrix},$$
where $\mat{G}$ is such that $\mat{Q}^T(\mat{G}D\mat{U}+\mat{U}^TD\mat{G}^T)\mat{Q}=C$.
To compute $\mat{G}$, one can first permute $C$ to get $C'=\mat{Q}C\mat{Q}^T$ (which remains
symmetric) and then use a call to \TRSS.

\subsubsection{Characteristic two case}\label{ssec:zlp}
In characteristic two, the equation $\mat{G}D\mat{U}+\mat{U}^TD\mat{G}^T=\mat{Q}C\mat{Q}^T$ in unknown $\mat{G}$ has in
general  no solution (as soon as $\mat{C}$ has a non-zero diagonal element).

However, one can still relax Problem~\ref{pb:trss} and allow the elimination to
leave a diagonal of elements not zeroed out.
Following Lemma~\ref{lem:cex}, the idea is then to
decompose $\mat{N}$ into a block tridiagonal form: 
$$\mat{N}=
\begin{bmatrix}\mat{P}&\mat{0}\\\mat{0}&\mat{Q}^T\end{bmatrix}
\begin{bmatrix}\mat{L}&\mat{0}\\\mat{G}&\mat{U}^T\end{bmatrix}
\begin{bmatrix}\mat{0}&\mat{D}\\\mat{D}&{\Delta}\end{bmatrix}
\begin{bmatrix}\mat{L}^T&\mat{G}^T\\\mat{0}&\mat{U}\end{bmatrix}
\begin{bmatrix}\mat{P}^T&\mat{0}\\\mat{0}&\mat{Q}\end{bmatrix},$$
where ${\Delta}$ is a diagonal matrix and now $\mat{G}$ is such that
$\mat{Q}^T(\mat{G}D\mat{U}+\mat{U}^T\mat{D}\mat{G}^T+\mat{U}^T{\Delta} \mat{U})\mat{Q}=C$.
Therefore ${\Delta}$ can be chosen such that the diagonal of 
$C''=\mat{Q}C\mat{Q}^T-\mat{U}^T{\Delta} \mat{U}=C'-\mat{U}^T{\Delta} \mat{U}$ is zero. 
As $\mat{U}$ is unit upper triangular, a simple pass over its
coefficients is sufficient to find such a ${\Delta}$: let
${\Delta}_{ii}=\mat{C'}_{ii}-\sum_{j=1}^{i-1}{\Delta}_{jj} \mat{U}_{j,i}^2$.
The algorithm is thus to permute $\mat{C}$ to get $\mat{C}'$;
then compute ${\Delta}$ with the recursive relation above
and update $\mat{C''}=\mat{C'}-\mat{U}^T{\Delta} \mat{U}$ with a \SYRDK. 
$\mat{C''}$ remains symmetric but with a zero diagonal and now \TRSS can be applied.

\subsection{The actual recursive algorithm}

\subsubsection{First phase: recursive elimination}
In the general case, the leading matrices are not full rank, and we have to
consider intermediate steps.
For the symmetric matrix $M\in\F^{(m+n){\times}(m+n)}$ of
Section~\ref{ssec:reccase}, its leading principal block $\mat{A}\in\F^{m{\times}m}$ is
of rank $r{\leq}m$.  
Thus its actual recursive decomposition is of the form:
$$\mat{A}=\mat{P}_1
\begin{bmatrix}\mat{L}_1\\\mat{M}_1\end{bmatrix}
\begin{bmatrix}\mat{D}_1\end{bmatrix}
\begin{bmatrix}\mat{L}_1^T&\mat{M}_1^T\end{bmatrix}
\mat{P}_1^T,$$
where $\mat{L}_1\in\F^{r{\times}r}$ is full rank
unit lower triangular, $\mat{D}_1\in\F^{r{\times}r}$ is block diagonal with
1 or 2-dimensional diagonal blocks, and $\mat{M}_1\in\F^{(m-r){\times}r}$.
Therefore, forgetting briefly the permutations, the decomposition of $M$
becomes: $$\mat{M}=
\begin{bmatrix}\mat{L}_1&\mat{0}&\mat{0}\\\mat{M}_1&\mat{I}&\mat{0}\\\mat{G}&\mat{0}&\mat{I}\end{bmatrix}
\begin{bmatrix}\mat{D}_1&\mat{0}&\mat{0}\\\mat{0}&\mat{0}&\mat{Y}\\\mat{0}&\mat{Y}^T&\mat{Z}\end{bmatrix}
\begin{bmatrix}\mat{L}_1^T&\mat{M}_1^T&\mat{G}^T\\\mat{0}&\mat{I}&\mat{0}\\\mat{0}&\mat{0}&\mat{I}\end{bmatrix},$$
where $\mat{Y}$ is such that 
$\mat{B}=\begin{bmatrix}\mat{L}_1\\\mat{M}_1\end{bmatrix}
\mat{D}_1\mat{G}^T+\begin{bmatrix}\mat{0}\\\mat{Y}\end{bmatrix}$.

From this point on, there remains to factorize the submatrix
$\begin{smatrix} \mat{0}&\mat{Y}\\\mat{Y}^T&\mat{Z}\end{smatrix}$. This will be carried out by the
algorithm described in the next section, working on a matrix with a zero 
leading principal submatrix.  
Supposing for now that this is possible, Algorithm~\ref{alg:recall} summarizes
the whole procedure.

\begin{algorithm}[htb]
\caption{Recursive symmetric indefinite elimination}\label{alg:recall}
\begin{algorithmic}[1]
  \REQUIRE $\mat{A}\in\F^{m{\times}m}$ and $\mat{C}\in\F^{n{\times}n}$ both symmetric,
  $\mat{B}\in\F^{m{\times}n}$.
  \ENSURE $\mat{P}$ permutation, $\mat{L}$ unit lower triangular, $\mat{D}$ block diagonal,
  s.t. $\begin{smatrix} \mat{A} & \mat{B} \\ \mat{B}^T & \mat{C}\end{smatrix}=\mat{P}\mat{L}\mat{D}\mat{L}^T\mat{P}^T$. 
\STATE Decompose $\mat{A}=\mat{P}_1
\begin{bmatrix}\mat{L}_1\\\mat{M}_1\end{bmatrix}
\begin{bmatrix}\mat{D}_1\end{bmatrix}
\begin{bmatrix}\mat{L}_1^T&\mat{M}_1^T\end{bmatrix}
\mat{P}_1^T$\;\hfill\COMMENT{Alg.~\ref{alg:recall}}
\STATE let $r=\rank(\mat{B})$ s.t. $\mat{L}_1, \mat{D}_1$ and $\mat{U}_1$ are $r\times r$.
\STATE $\mat{B'}=\mat{P}_1^T\mat{B}$\;\hfill\COMMENT{\PERM$(\mat{P}_1^T,\mat{B})$}
  \STATE  Split $\mat{B'} =\begin{bmatrix}\mat{B'}_1\\ \mat{B}'_2\end{bmatrix}$ where $\mat{B'}_1$ is
  $r\times n$.
  \STATE $X\leftarrow{}\mat{L}_1^{-1}\mat{B'}_1$\;\hfill\COMMENT{\TRSM($\mat{L}_1,\mat{B}_1'$)}
  \STATE $\mat{Y}\leftarrow{}\mat{B'}_2-\mat{M}_1\mat{X}$\;\hfill\COMMENT{\GEMM($\mat{B}_2',\mat{M}_1,\mat{X}$)}
  \STATE $\mat{G}\leftarrow{}X^T\mat{D}_1^{-1}$\;\hfill\COMMENT{\SCAL$(X^T,\mat{D}_1^{-1})$}
  \STATE $\mat{Z}\leftarrow{}C-\mat{G}\mat{D}_1\mat{G}^T$\;\hfill\COMMENT{\SYRDK($C,\mat{G},\mat{D}_1$)}
  \STATE Decompose
  $\begin{bmatrix} \mat{0} & \mat{Y}\\ \mat{Y}^T & \mat{Z}\end{bmatrix}=
  \mat{P}_2\mat{L}_2\mat{D}_2\mat{L}_2^T\mat{P}_2^T$\;\hfill\COMMENT{Alg.~\ref{alg:zlp}}
  \STATE $P\leftarrow{}\begin{smatrix}\mat{P}_1&0\\0&\mat{I}_n\end{smatrix} \cdot{}
  \begin{smatrix}\mat{I}_{r}&0\\0&\mat{P}_2\end{smatrix}$\;
    \STATE $\mat{N}_1 \leftarrow \mat{P}_2^T \begin{bmatrix} \mat{M}_1\\\mat{G}\end{bmatrix}$\hfill\COMMENT{\PERM$(\mat{P}_2^T, \begin{bmatrix}  \mat{M}_1\\\mat{G}  \end{bmatrix}  )$}
    \STATE $L\leftarrow{}\begin{bmatrix}[c|c]\mat{L}_1&\\\hline\begin{matrix}
      \mat{N}_1\end{matrix}&\mat{L}_2\end{bmatrix}$\; 
  \STATE $D\leftarrow{}\begin{bmatrix}[c|c]\mat{D}_1&\\\hline&\mat{D}_2\end{bmatrix}$\;
\end{algorithmic}
\end{algorithm}

\subsubsection{Second phase: off-diagonal pivoting}

Consider $\mat{N}=\begin{smatrix} \mat{0} & \mat{B}\\ \mat{B}^T & \mat{C}\end{smatrix}$, where
$\mat{B}\in\F^{m{\times}n}$, with $m{\leq}n$, has now an arbitrary rank
$r{\leq}m$. Then its
PLDUQ decomposition is of the form
$$\mat{B}=\mat{P}
\begin{bmatrix}\mat{L}_1\\\mat{M}_1\end{bmatrix}
\begin{bmatrix}\mat{D}_1\end{bmatrix}
\begin{bmatrix}\mat{U}_1&\mat{V}_1\end{bmatrix}
\mat{Q},
$$
with $\mat{D}_1$ diagonal, and $\mat{L}_1$ and $\mat{U}_1$ unit square triangular matrices, all
three of order~$r$.
Then consider a conformal block decomposition of
$\mat{Q}\mat{C}\mat{Q}^T=\begin{smatrix}\mat{C}_1&\mat{C}_2\\\mat{C}_2^T&\mat{C}_3\end{smatrix}$  where $\mat{C}_1$ is $r\times r$.
It remains to eliminate $\mat{C}_1$ and $\mat{C}_2$ with the pivots found in $B$, which
leads to the following factorization:
\begin{equation}\label{eq:zlp}
\begin{split}
\mat{N}=
\begin{bmatrix} \mat{P}_1\\&\mat{Q}_1^T\end{bmatrix}
\begin{bmatrix}\mat{L}_1&  0& 0\\\mat{M}_1&  0& 0\\\mat{G}_1&  \mat{U}_1^T& 0\\\mat{G}_2&\mat{V}_1^T&\mat{I}\end{bmatrix}
\begin{bmatrix}0&\mat{D}_1&0\\ \mat{D}_1&0&0\\0&0&\mat{Z}\end{bmatrix} \times \\
\begin{bmatrix}\mat{L}_1^T&\mat{M}_1^T&\mat{G}_1^T&\mat{G}_2^T\\0&0&\mat{U}_1&\mat{V}_1\\0&0&0&\mat{I}\end{bmatrix}
  \begin{bmatrix} \mat{P}_1^T\\&\mat{Q}_1\end{bmatrix}
    \end{split}
\end{equation}

where $\mat{G}_1$ satisfies \begin{equation}\label{eq:G1}
  \mat{U}_1^T\mat{D}_1\mat{G}_1^T+\mat{G}_1\mat{D}_1\mat{U}_1=\mat{C}_1
\end{equation}
  and
$\mat{G}_2=(\mat{C}_2-\mat{V}_1^T\mat{D}_1\mat{G}_1^T)\mat{U}_1^{-1}\mat{D}_1^{-1}$ and $\mat{Z}=\mat{C}_3-(\mat{V}_1^T\mat{D}_1\mat{G}_2^T+\mat{G}_2\mat{D}_1\mat{V}_1)$.

In order to produce a LDLT decomposition, there still remains to perform
permutations to
\begin{compactenum}
\item compact the leading elements of the lower triangular matrix into a
  $2r\times 2r$ invertible leading triangular submatrix,\label{enum:perm:rotations}
\item make the $  \begin{smatrix}    &\mat{D}_1\\\mat{D}_1  \end{smatrix}$ matrix block
  diagonal with 1 or 2-dimensional diagonal blocks.\label{enum:perm:interleave}
\end{compactenum}

The permutation matrix \begin{equation}\label{eq:p1}
\mat{P}_c=\begin{smatrix}\mat{I}_r& 0& 0&0\\0&0& \mat{I}_{m-r}&0\\0&\mat{I}_r&0&0\\0&0&0&\mat{I}_{n-r}\end{smatrix},
\end{equation}
corresponding to a block circular rotation, takes care of
condition~\ref{enum:perm:rotations}, while preserving precedence in the
non-pivot rows. This is a requirement for the factorization to reveal the rank
profile matrix~\cite{jgd:2017:bruhat}.
The decomposition becomes
\begin{equation}
\begin{split}
N=
\begin{bmatrix} \mat{P}_1\\&\mat{Q}_1^T\end{bmatrix} \mat{P}_c
\begin{bmatrix}\mat{L}_1&  0& 0\\\mat{G}_1&  \mat{U}_1^T& 0\\\mat{M}_1&  0& 0\\\mat{G}_2&\mat{V}_1^T&\mat{I}\end{bmatrix}
\begin{bmatrix}0&\mat{D}_1&0\\ \mat{D}_1&0&0\\0&0&\mat{Z}\end{bmatrix} \times \\
\begin{bmatrix}\mat{L}_1^T&\mat{G}_1^T&\mat{M}_1^T&\mat{G}_2^T\\0&\mat{U}_1&0&\mat{V}_1\\0&0&0&\mat{I}\end{bmatrix}
\mat{P}_c^T  \begin{bmatrix} \mat{P}_1^T\\&\mat{Q}_1\end{bmatrix}
    \end{split}
\end{equation}



In order to achieve Condition~\ref{enum:perm:interleave}, we will transform the
matrix $\begin{smatrix} 0&\mat{D}_1\\\mat{D}_1&0\end{smatrix}$ into the block diagonal
  matrix $\text{Diag}(\begin{smatrix} 0&d_i\\d_i&0\end{smatrix})$ where $d_i$ is
    the $i$th diagonal element in $\mat{D}_1$.
To describe the process, we will focus on the matrix
 $$\mat{N}_2=
\begin{bmatrix}\mat{L}_1& 0\\\mat{G}_1& \mat{U}_1^T\end{bmatrix}
\begin{bmatrix}0&\mat{D}_1\\\mat{D}_1&0\end{bmatrix}
\begin{bmatrix}\mat{L}_1^T&\mat{G}_1^T\\0&\mat{U}_1\end{bmatrix}
=\overline{L}\cdot\overline{\Delta}\cdot\overline{L}^T,$$
and consider a splitting in halves of the matrix
$\mat{D}_1=\begin{smatrix}  \mat{D}_{11}\\&\mat{D}_{12}\end{smatrix}$ where $\mat{D}_{11}$ has order
$r_1$ and $\mat{D}_{12}$ order $r_2$. This leads to the conformal decompostion
{\small
$$\begin{bmatrix}\mat{L}_{11}&0&0&0\\\mat{L}_{12}&\mat{L}_{13}&0&0\\\mat{G}_{11}&\mat{G}_{14}&\mat{U}_{11}^T&0\\\mat{G}_{12}&\mat{G}_{13}&\mat{U}_{12}^T& \mat{U}_{13}^T\end{bmatrix}
\begin{bmatrix}0&0&\mat{D}_{11}&0\\0&0&0&\mat{D}_{12}\\\mat{D}_{11}&0&0&0\\0&\mat{D}_{12}&0&0\end{bmatrix}
\begin{bmatrix}\mat{L}_{11}^T&\mat{L}_{12}^T&\mat{G}_{11}^T&\mat{G}_{12}^T\\0&\mat{L}_{13}^T&\mat{G}_{14}^T&\mat{G}_{13}^T\\0&0&\mat{U}_{11}&\mat{U}_{12}\\0&0&0&\mat{U}_{13}\end{bmatrix}
$$
}
Then considering the permutation matrix
$$\mat{P}_d=\begin{smatrix}\mat{I}_{r_1}& 0& 0&0\\0&0& \mat{I}_{r_2}&0\\0&\mat{I}_{r_1}&0&0\\0&0&0&\mat{I}_{r_2}\end{smatrix},$$
one  can form 
$\mat{P}_d^T\overline{\Delta}\mat{P}_d=\begin{smatrix}0&\mat{D}_{11}&0&0\\\mat{D}_{11}&0&0&0\\0&0&0&\mat{D}_{12}\\0&0&\mat{D}_{12}&0\end{smatrix}$
and
$\mat{P}_d^T\overline{L}\mat{P}_d=
\begin{smatrix}\mat{L}_{11}&0&0&0\\\mat{G}_{11}&\mat{U}_{11}^T&\mat{G}_{14}&0\\\mat{L}_{12}&0&\mat{L}_{13}&0\\\mat{G}_{12}&\mat{U}_{12}^T&\mat{G}_{13}&
  \mat{U}_{13}^T\end{smatrix}$.
  Applying this process recursively changes $\overline \Delta$ into the desired
  block diagonal form. Then the transformation of $\overline L$ will remain
  lower triangular if and only if all $\mat{G}_{14}$ matrices are zero: this means
  that  $\mat{G}_1$ must be lower triangular in the first place.

Finding $\mat{G}_1$ lower triangular satifying Equation~\eqref{eq:G1}, is an instance
of Problem~\ref{pb:trss} for which the routine \TRSS provides a solution.

Note that the actual permutation to transform 
$\begin{smatrix}\mat{0}&\mat{D}_1\\\mat{D}_1&\mat{0}\end{smatrix}$ into a $2{\times}2$-blocks diagonal
matrix is a permutation matrix, $\mat{P}_i$, resulting from the one by one
interleaving of the rows of $\begin{smatrix}\mat{I}_r&\mat{0}\end{smatrix}$ and
$\begin{smatrix}\mat{0}&\mat{I}_r\end{smatrix}$. If
$\mat{e}_i=\begin{smatrix}0\ldots{}0&1&0\ldots{}0\end{smatrix}^T$ is the $i$-th canonical vector, then:
\begin{equation}\label{eq:p2}
\mat{P}_i=\begin{bmatrix}\mat{e}_1&\mat{e}_{r+1}&\mat{e}_2&\mat{e}_{r+2}&\ldots&\mat{e}_{r}&\mat{e}_{2r}\end{bmatrix}.
\end{equation}
Similarly the triangular factor of the factorization is thus a one by one
interleaving of the rows of $\begin{smatrix}\mat{L}_1&\mat{0}\end{smatrix}$ and
$\begin{smatrix}\mat{G}_1&\mat{U}_1^T\end{smatrix}$ as well as a one by one interleaving of
the columns $\begin{smatrix}\mat{L}_1\\\mat{G}_1\end{smatrix}$ and
$\begin{smatrix}\mat{0}\\\mat{U}_1^T\end{smatrix}$, which overall remains triangular.

Finally, a call to Algorithm~\ref{alg:recall} produces a factorization for the
remaining $Z$ block and a final block rotation,  
$$\begin{bmatrix}\mat{I}_{2r}&\mat{0}&\mat{0}\\\mat{0}&\mat{0}&\mat{I}_{m-r}\\\mat{0}&\mat{I}_{n-r}&\mat{0}\end{bmatrix},$$ 
moves the intermediate zero rows and columns to the bottom right.
The full algorithm is presented in details in Algorithm~\ref{alg:zlp} (for zero
or odd characteristic, the characteristic two case being presented afterwards in
Section~\ref{ssec:even}).

\begin{algorithm}[htbp]
\caption{Rank deficient and zero leading principal symmetric elimination}\label{alg:zlp}
\begin{algorithmic}[1]
  \REQUIRE $\mat{C}\in\F^{n{\times}n}$ symmetric and $\mat{B}\in\F^{m{\times}n}$.
  \ENSURE $\mat{P}$ permutation, $L$ unit lower triangular, $\mat{D}$ block-diagonal, 
  s.t. $\begin{bmatrix} \mat{0} & \mat{B}\\ \mat{B}^T & \mat{C}\end{bmatrix}=\mat{P}\mat{L}\mat{D}\mat{L}^T\mat{P}^T$.
  \STATE Decompose $\mat{B}=\mat{P}_B
\begin{bmatrix}\mat{L}_1\\\mat{M}_1\end{bmatrix}
\begin{bmatrix}\mat{D}_1\end{bmatrix}
\begin{bmatrix}\mat{U}_1&\mat{V}_1\end{bmatrix}
Q$\;\hfill\COMMENT{\PLDUQ}
  \STATE
  $\mat{C'}\leftarrow \mat{Q}\mat{C}\mat{Q}^T=\begin{bmatrix}\mat{C'}_1&\mat{C'}_2\\{\mat{C'}_2}^T&\mat{C'}_3\end{bmatrix}$\;\hfill\COMMENT{$\mat{C'}_1$ first $r=rk(\mat{B})$ rows/columns: \PERM}
  \IF{characteristic $(\F) =2$}
    \FOR{$i=1$ \TO $r$}
      \STATE $\Delta_{ii}=\left(\mat{C'}_1\right)_{ii}-\sum_{j=1}^{i-1} \Delta_{jj} \left(\mat{U}_1\right)_{j,i}^2$\;
    \ENDFOR
    \STATE $\mat{C'}_1\leftarrow{}\mat{C'}_1-\mat{U}_1^T{\Delta}\mat{U}_1$\;\hfill\COMMENT{\SYRDK($\mat{C}_1',\mat{U}_1,\Delta$)}
  \ENDIF
  \STATE $\mat{X}^T\mat{U}_1+\mat{U}_1^T\mat{X}=\mat{C'}_1$\;\hfill\COMMENT{\TRSS($\mat{U}_1,\mat{C'}_1$)}
  \STATE $\mat{G}_1\leftarrow{}\mat{X}^T\mat{D}_1^{-1}$\;\hfill\COMMENT{\SCAL$(\mat{X}^T,\mat{D}_1^{-1})$}
  \IFTHEN{characteristic$(\F)=2$}{$\mat{X}\leftarrow \mat{X}+\Delta\mat{U}_1$}\;\hfill\COMMENT{\DADD($X$,$\Delta$,$\mat{U}_1$)}
  \STATE $\mat{C''}_2\leftarrow{}\mat{C'}_2-X^T\mat{V}_1$\;\hfill\COMMENT{\TRMM($X^T,\mat{V}_1$)}
  \STATE $\mat{Y}\leftarrow{}\mat{U}_1^{-T}\mat{C''}_2$\;\hfill\COMMENT{\TRSM($\mat{U}_1^T,\mat{C''}_2$)}
  \STATE $\mat{Z}\leftarrow{}\mat{C'}_3-(\mat{Y}^T\mat{V}_1+\mat{V}_1^T\mat{Y})$\;\hfill\COMMENT{\SYRDtwoK($\mat{Y},\mat{V}_1$)}
  \IFTHEN{characteristic$(\F)=2$}{$\mat{Z}\leftarrow{}\mat{Z}-(\mat{V}_1^T\Delta{}\mat{V}_1)$}\;\hfill\COMMENT{\SYRDK($\mat{Z},\mat{V}_1,\Delta$)}
  \STATE $\mat{G}_2\leftarrow{}\mat{Y}^T\mat{D}_1^{-1}$\;\hfill\COMMENT{\SCAL$(\mat{Y}^T,\mat{D}_1^{-1})$}
  \STATE Decompose
  $\mat{Z}=\mat{P}_3\mat{L}_3\mat{D}_3\mat{L}_3^T\mat{P}_3^T$\;\hfill\COMMENT{Alg.~\ref{alg:recall}}
  \\\COMMENT{With $\mat{P}_c$ from (\ref{eq:p1}), and $\mat{P}_i$ from (\ref{eq:p2}):}  
  \STATE $P\leftarrow\begin{smatrix}\mat{P}_B&0\\0&Q^T\end{smatrix}
  \begin{smatrix}\mat{P}_c&0\\0&\mat{I}_{n-r}\end{smatrix}
  \begin{smatrix}\mat{P}_i&0\\0&\mat{I}_{m+n-2r}\end{smatrix}
  \begin{smatrix}\mat{I}_{m+r}&0\\0&\mat{P}_3\end{smatrix}
  \begin{smatrix}\mat{I}_{2r}&0&0\\0&0&\mat{I}_{m-r}\\0&\mat{I}_{n-r}&0\end{smatrix}$\;
  \STATE $L\leftarrow\begin{bmatrix}[c|c]\mat{P}_i^T\begin{smatrix}\mat{L}_1&\\\mat{G}_1&\mat{U}_1^T\end{smatrix}\mat{P}_i&\\\hline
  \begin{matrix}\mat{G}_2&\mat{V}_1^T \\ \mat{M}_1&0\end{matrix}&
  \begin{matrix}\mat{L}_3\\ 0\; \end{matrix}\end{bmatrix}$\;
  \STATE $D\leftarrow\begin{bmatrix}[c|c] \mat{P}_i^T\begin{smatrix}0&\mat{D}_1\\\mat{D}_1&0\end{smatrix}\mat{P}_i&\\\hline&\mat{D}_3\end{bmatrix}$\;
\end{algorithmic}
\end{algorithm}

\subsection{Characteristic two}\label{ssec:even}
The case of the characteristic two can be handled similarly, just computing the extra
diagonal and updating after the \PLDUQ decomposition, as sketched in
Section~\ref{ssec:zlp}. 
Indeed, the only issue is the division by $2$ in \TRSS, which is removed if the
diagonal of $\mat{C'}_1$ is zero. 
Therefore, Algorithm~\ref{alg:recall} is unchanged, the block diagonal matrix
just has lower symmetric antitriangular $2{\times}2$ blocks instead of only
antidiagonal ones. The only few additional operations appear in 
Algorithm~\ref{alg:zlp}  and are the contents of the
"{\algorithmicif{characteristic$(\F)=2$}}$\ldots$" branchings.

Then the tridiagonal form with symmetric antitriangular
${2\times}2$ blocks thus obtained by Algorithm~\ref{alg:zlp} can be used to
either reveal the rank profile matrix
(via computing $\Psi$, the support matrix of $D$, and the pivoting matrix
$\RPM{}=P\Psi P^T$) or a PLDLTPT factorization, both at an
extra linear cost, as shown in Section~\ref{ssec:antiti}.

Overall, we have proven:
\begin{theorem} Algorithm~\ref{alg:recall} correctly computes a symmetric
  indefinite PLDLTPT factorization revealing the rank profile matrix. 
\end{theorem}

\section{Base case iterative variant}\label{sec:basecase}

The recursion of Algorithm~\ref{alg:recall} should not be performed all the way
to a dimension 1 in practice.
For implementations over a finite field, it would induce an unnecessary large
number of modular reductions and a significant amount of data movement for the
permutations. Instead, we propose in Algorithm~\ref{alg:BC} an iterative algorithm computing a PLDLTPT
revealing the rank profile matrix to be used as a base case in the recursion.

This iterative algorithm has the following features:
\begin{compactenum}
  \item it uses a pivot search minimizing the lexicographic order
(following the caracterization in~\cite{jgd:2017:bruhat}): if the diagonal element
    of the current row is 0, the pivot is chosen as the first non-zero element
    of the row, unless the row is all zero, in which case, it is searched in the
    following row;
  \item the pivot is permuted with cyclic shifts on the row and columns, so as
    to leave the precedence in the remaining rows and columns unchanged.
  \item the update of the unprocessed part in the matrix is delayed following
    the scheme of a Crout elimination 
    schedule~\cite{DDSV98}. It does not only improves efficiency thanks to a
    better data locality, but it also reduces the amount of modular reductions,
    over a finite field, as shown for the unsymmetric case
    in~\cite{DGPRZ16}. 
\end{compactenum}

We denote by $\rho_{i,n}$ the cyclic shift permutation of order $n$ moving
element $i$ to the first position: $\rho_{i,n}=(i,0,1,\dots, i-1,i+1,\dots n-1)$.
Indices are 0 based, index ranges are excluding their upper bound. For instance,
$\mat{A}_{i,0..r}$ denotes the $r$ first elements of the $i+1$st row of $\mat{A}$, and
$\mat{A}_{0..r,0..r}$ is the 0-dimensional matrix when $r=0$.

\begin{algorithm}[htbp]
  \caption{SYTRF Crout iterative base case}\label{alg:BC}
  \begin{algorithmic}[1]
    \REQUIRE{$\mat{A} \in \F^{n\times n}$ symmetric}
    \ENSURE{$\mat{P}$, a permutation, $\mat{L}$, unit lower triangular and
      $\mat{D}$, block diagonal,  such that $\mat{A}=\mat{P}\mat{L}\mat{D}\mat{L}^T\mat{P}^T$}
    \STATE $r\leftarrow 0$; $\mat{D}\leftarrow \mat{0}$\;\hfill\COMMENT{Denote $\mat{W}=\mat{A}$ the working matrix}
    \FOR{$i=0..n$}
      \STATE Here $\mat{W}=\begin{bmatrix}  \mat{L} \\  \mat{M} & 0 & 0\\ \mat{N} & 0 & \mat{A}_{i..n,i}
      & \mat{A}_{i..n,i+1..n}  \end{bmatrix}$ with $\mat{L}\in\F^{r{\times}r}$
      \STATE $\mat{v}\leftarrow \mat{N}_{0,0..r}\times \mat{D}_{0..r,0..r}^{-1}$
      \STATE $\mat{c}\leftarrow \mat{A}_{i..n,i} -  \mat{N}\times \mat{v}^T$
      \IFTHEN{$\mat{c}=0$}{Loop to next iteration}
      \STATE Let $j$ be the smallest index such that $x=c_j\neq 0$
      \IF[Denote $\mat{c}=\lbrack\mat{0}\;x\;\mat{k}\rbrack^T$]{$j=0$}
        \STATE $ \begin{smatrix} \mat{M}\\\mat{N}  \end{smatrix} \leftarrow
        \rho_{j,n-r}\times \begin{smatrix} \mat{M}\\ \mat{N}  \end{smatrix}$
        \STATE $\mat{W}_{r..n,r}\leftarrow x^{-1}\times \mat{\rho}_{j,n-r}\times    \begin{smatrix}       0\\ \mat{c}    \end{smatrix}$
        \STATE $\mat{P}\leftarrow \mat{P} \times  \mat{\rho}_{j,n-r}^T$
        \STATE $\mat{D}_{r,r}\leftarrow x$
        \STATE $r\leftarrow r+1$
      \ELSE[Crout update of the row $i+j$]  
        \STATE $\mat{w}\leftarrow \mat{N}_{j,0..r}\times \mat{D}_{0..r,0..r}^{-1}$
        \STATE $\mat{d}\leftarrow \mat{A}_{i..n,j+i} -  \mat{N}\times \mat{w}^T (=
        \begin{bmatrix} \mat{0} & x & \mat{g}& y & \mat{h} \end{bmatrix}^T
        )$
        \STATE Here $\mat{W}=\begin{bmatrix}  \mat{L} \\
        \mat{M} & 0 & 0& 0&0&0 \\
        \multirow{4}{*}{$\mat{N}$} &0 & 0& 0&x&\mat{k}^T\\
        &0& 0 & \mat{F} & \mat{g}&\mat{J}^T \\
        &0&x & \mat{g}^T & y &\mat{h}^T\\
        & 0 & \mat{k} &\mat{J}&\mat{h}&*\end{bmatrix}$
        \IF{$\text{characteristic}(\F)=2$}
          \STATE $y'\leftarrow 0$
          \STATE $\mat{h'}\leftarrow\mat{h} - yx^{-1} \mat{k}$
          \STATE $\mat{D}_{r..r+2,r..r+2} \leftarrow   \begin{bmatrix} 0&x\\x&y   \end{bmatrix}$
        \ELSE
          \STATE $y'\leftarrow y/2$
          \STATE $\mat{h'}\leftarrow\mat{h} - y'x^{-1} \mat{k}$
          \STATE $\mat{D}_{r..r+2,r..r+2} \leftarrow   \begin{bmatrix} 0&x\\x&0   \end{bmatrix}$
        \ENDIF
        \STATE Perform cyclic symmetric row and column rotations to bring $W$ to
        the form
        $\mat{W}=\left[\begin{array}{ccc|ccc}  \mat{L} & &\\
        \multirow{2}{*}{$\mat{n}'$} &1&  \\
        &x^{-1}y' &1 \\
        \hline
        \mat{M}&0  & 0&0 &0&0\\
        \multirow{2}{*}{$\mat{N}'$}&x^{-1} \mat{g}  & 0 &0& \mat{F}&\mat{J}^T\\
        &x^{-1}\mat{h'}  & x^{-1}\mat{k}&0& \mat{J} &*
        \end{array}\right]$
        \STATE Update $\mat{P}$ accordingly
        \STATE $r\leftarrow r+2$
        \ENDIF
    \ENDFOR
  \end{algorithmic}
\end{algorithm}

\section{Experiments}\label{sec:expe}
We now report on experiments of an implementation of these algorithms in the
FFLAS-FFPACK library~\cite{FFLAS18}, dedicated to dense linear algebra over finite
fields. We used the version committed under the reference
\href{https://github.com/linbox-team/fflas-ffpack/commit/e12a9989f07a9e128d9b4dd59681e8f62e8fe3b1}{e12a998}
of the master branch. It was compiled with gcc-5.4 and was linked with the
numerical library OpenBLAS-0.2.18. Experiments are run on a single core of an an
Intel Haswell i5-4690, \symbol{64}3.5GHz.

Computation speed are normalized as \textit{effective Gfops}, an estimate of the
number of field operations that an algorithm with classic matrix arithmetic
would perform per second, divided by the computation time. For a matrix of order $n$
and rank $r$, we defined this as:
$$\text{Effective Gfops} = (r^3/3+n^2r-r^2n) / (10^9 \times\text{time}).$$
All experiments are over the 23-bits finite field $\Z/8388593\Z$.

Figure~\ref{fig:exp:BC} compares the computation speed of the pure recursive
algorithm, the base case algorithm and a cascade of these two, with a threshold
set to its optimum value from experiments on this machine.
Remark that the pure recursive variant performs rather well with generic rank
profile matrices, while matrices with uniformly random rank profile matrix make
this variant very slow, due to an excessive amount of pivoting.
As expected, the base case Crout variant speeds up these instances for small dimensions, but
then its performance stagnate on large dimensions, due to poor cache efficiency. Lastly the
cascade algorithm combines the benefits of the two variants and therefore
performs best in all settings.
We here used a threshold $n=128$ for the experiments with random RPM matrices, but of
only $n=48$ for generic rank profile matrices, since the recursive variant becomes
competive much earlier. In most cases, the rank profile structure of given
matrices is unknown a priori, making the setting of this threshold
speculative. One could instead implement an introspective strategy, updating
the threshold from experimenting with running instances.
\begin{figure}[htbp]
\centering%
\includegraphics{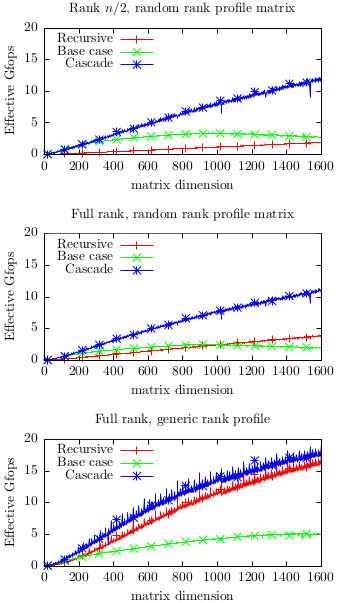}
\caption{
 Computation speed of the Base Case, the pure recurisve and the cascading
 variant for a rank profile matrix revealing PLDLTPT decomposition. Matrices with rank half the dimension and random RPM  (top), full rank with random RPM
 (center) or full rank with generic rank profile (bottom)}\label{fig:exp:BC}%
\end{figure}%

\begin{table}[htb]\centering
  \begin{tabular}{cllllll}
  \toprule
  \multirow{3}{*}{$n$} & \multicolumn{2}{c}{Gen. rank prof.} &
  \multicolumn{2}{c}{Random RPM} & \multicolumn{2}{c}{Random RPM}\\
  &\multicolumn{2}{c}{ $r=n$}&\multicolumn{2}{c}{ $r=n$}&\multicolumn{2}{c}{ $r=n/2$}\\
  & PLUQ & LDLT & PLUQ& LDLT & PLUQ & LDLT\\
  \midrule
  $100$& 5.81e-4 & 4.95e-4   &6.71e-4 & 5.95e-4 & 3.79e-4 & 3.69e-4\\
  $200$& 2.29e-3 & 1.25e-3   &3.05e-3 &1.82e-3 &1.81e-3 & 1.23e-3\\
  $500$& 1.99e-2  & 6.57e-3   &3.07e-2 & 1.05e-2&2.04e-2 & 7.54e-3\\
  $1000$& 0.104  & 2.58e-2    &1.15e-1 & 4.25e-2 &6.98e-2 & 3.14e-2 \\
  $2000$& 0.507  & 0.134     &0.551 & 0.199&0.308 & 0.148 \\
  $5000$& 4.651  & 1.720     &4.502 & 2.003 &2.813 & 1.419 \\
  $10000$&26.59 & 11.94 & 26.08 &15.88 &12.04 & 8.265\\
  \bottomrule
    
  \end{tabular}
  \caption{Comparing computation time (s) of the symmetric (LDLT) with
    unsymmetric (PLUQ) triangular decompositions. Matrices with rank $r$, and rank profile matrix uniformly random. }
\label{tab:timings}
   
\end{table}

Table~\ref{tab:timings} compares the computation time of the symmetric
decomposition algorithm with that of the unsymmetric case (running the PLUQ
algorithm of~\cite{jgd:2017:bruhat}). These experiments confirm a speed-up factor
of about 2 between these routines, which is the expected gain in the constant in
the time complexity. Note that on large instances, the PLUQ elimination performs
better with random RPM instances than generic rank profiles, contrarily to the
LDLT routine. This is due to the lesser amount of arithmetic operations when the
RPM is random (some intermediate submatrices being rank deficient). On the other
hand, these matrices generate more off-diagonal pivots, which cause more
pivoting in LDLT than in PLUQ, explaining the slow down for the symmetric case.




\bibliographystyle{plainurl}
\bibliography{ldlt}

\end{document}